\documentclass[12pt, oneside, a4paper]{article}

\usepackage{amsmath}
\usepackage{amsfonts}
\usepackage{amssymb}
\usepackage{amsthm,mathrsfs}
\usepackage{enumerate}
\usepackage{graphicx}
\newtheorem{theorem}{Theorem}[section]
\newtheorem{conjecture}[theorem]{Conjecture}
\newtheorem{lemma}[theorem]{Lemma}

\newtheorem{corollary}[theorem]{Corollary}

\theoremstyle{definition}
\newtheorem{remark}[theorem]{Remark}

\title{\textbf{Cayley graphs on $p$-solvable groups generated by $p$-singular elements}}
\author{Mahdi Ebrahimi\footnote{ m.ebrahimi.math@ipm.ir}
 \\
 {\small\em  School of Mathematics, Institute for Research in Fundamental Sciences (IPM)},\\{\small\em P.O. Box: 19395--5746, Tehran, Iran},\\{\small\em ORCID ID: 0000-0001-9789-7376}\\ Declarations of interest: none
\\
}
\date{}

\begin{document}

\maketitle


\begin{abstract}
For a graph $\Gamma$, the multiplicity of the eigenvalue $0$, denoted by $\eta(\Gamma)$, is called the nullity of $\Gamma$.
 Also the energy of $\Gamma$, denoted by $\mathcal{E}(\Gamma)$, is defined as the sum of the absolute values of the eigenvalues of $\Gamma$.
The index of a subgroup $H$ in a group $G$ is denoted by $[G:H]$.
 For a prime $p$, let $G$ be a finite $p$-solvable group whose order is divisible by $p$. Also let $\Omega_p(G)$ be the set of all $p$-singular elements of $G$.
  In this paper, we apply block theory of finite groups to show that the Cayley graph $\Gamma_p(G):=\mathrm{Cay}(G,\Omega_p(G))$
   is an integral graph with $\eta(\Gamma_p(G))=|G|-[G:O_{p^\prime}(G)]$, where $O_{p^\prime}(G)$ is  the largest normal subgroup of $G$ whose order is co-prime to $p$.
   We also find a lower bound for $\mathcal{E}(\Gamma_p(G))$.
     Finally, we prove that the diameter of $\Gamma_p(G)$ is at most $ |G|_p$.

  \end{abstract}

\section{Introduction}
$\noindent$In this paper, all groups and graphs are assumed to be finite. Let $\Gamma$ be a simple graph with vertex set $\{\nu_1,\nu_2,\dots, \nu_n\}$. The \textit{adjacency matrix} of $\Gamma$, denoted by $A(\Gamma)$, is the $n\times n$ matrix such that the $(i,j)$-entry is $1$ if $\nu_i$ and $\nu_j$ are adjacent, and is $0$ otherwise. The \textit{eigenvalues} of $\Gamma$ are the eigenvalues of its adjacency matrix $A(\Gamma)$. The study of eigenvalues of graphs is an important part of modern graph theory.
The graph $\Gamma$ is called  \textit{singular}, whenever the matrix $A(\Gamma)$ is singular. It is clear that the graph $\Gamma$ is singular if and only if $0$ is an eigenvalue of it.
For the graph $\Gamma$, the multiplicity of the eigenvalue $0$, denoted by $\eta(\Gamma)$, is called the \textit{nullity} of $\Gamma$.
The problem of graph singularity was first appeared in structural chemistry in the context of Huckel Theory \cite{Huckel}. Singular graphs play a significant role in graph theory, and there are many applications in physics and chemistry (see \cite{ph}, \cite{sz} and \cite{ch}). The chemical importance of the nullity of graphs lies in the fact that if $\eta(\Gamma)$ is non-zero for a molecular graph $\Gamma$, then the corresponding chemical compound is highly reactive and unstable, or nonexistent (see \cite{At}, \cite{Cvg}).

Suppose $\{\lambda_1, \lambda_2, \dots, \lambda_n\}$ is the set of all eigenvalues of $\Gamma$. The \textit{energy} of the graph $\Gamma$ is defined as $\mathcal{E}(\Gamma):=\sum_{i=1}^n|\lambda_i|$. This concept was introduced by Gutman \cite{182}. In recent decades, the energy of a graph has been much studied in the mathematical-chemistry literature for instance see, \cite{Ak}, \cite{7}, \cite{GZ}, \cite{oboudi}  and \cite{zh}. The total $\pi$-electron energy and various "resonance energies" derived from it, play an important role in the theory of conjugated molecules, for more details see \cite{8}.
In recent years, several lower bounds for the energy of graphs have been achieved by researchers. The reader should be referred, for instance to \cite{Ak}, \cite{Bo}, \cite{jahan}, \cite{Ma}, \cite{oboudi} and \cite{Ti}.

In this paper, we assume that the reader is familiar with finite group theory (see \cite{Is}). Let $G$ be a finite group and $S$ be
 an inverse closed subset of $G$ with $1 \notin S$.
  The \textit{Cayley graph} $\mathrm{Cay}(G,S)$ is the graph which has the elements of
   $G$ as its vertices and two vertices $u,\nu \in G$
    are joined by an edge if and only if $\nu=au$, for some $a\in S$. A Cayley graph $\mathrm{Cay}(G,S)$ is called \textit{normal} if $S$ is closed under conjugation with elements of $G$. Also $\mathrm{Cay}(G,S)$ is called \textit{integral} if its eigenvalues are all integers.

Now let $p$ be a prime. The $p$-part and $p^{\prime}$-part of the order of $G$ are denoted by $|G|_p$ and $|G|_{p^{\prime}}$, respectively. The \textit{$p^\prime$-core} of $G$ is the largest normal subgroup of $G$ whose order is co-prime to $p$ and is denoted by $O_{p^\prime}(G)$. For a subgroup $H$ of $G$, the index of $H$ in $G$ is denoted by $[G:H]$. We denote by $r_p(G)$, the index $[G:O_{p^\prime}(G)]$ of $O_{p^\prime}(G)$ in $G$. The set of prime divisors of the order of $G$ is denoted by $ \pi(G)$. Suppose $p\in \pi(G)$. A \textit{$p$-singular element} of $G$ is an element $g\in G$ in which the order of $g$ is divisible by $p$. We use the notations $\Omega_p(G)$ and $H_p(G)$ for the set of all $p$-singular elements of $G$, and the subgroup of $G$ generated by $\Omega_p(G)$, respectively. We also denote by $c_p(G)$ and $d_p(G)$, the index $[G:H_p(G)]$ and the cardinality of $\Omega_p(G)$, respectively. Assume that $\Gamma_p(G):=\mathrm{Cay}(G,\Omega_p(G))$. In this paper, we wish to study the nullity and energy of $\Gamma_p(G)$, whenever $G$ is a  $p$-solvable group.

 \begin{theorem}\label{energy}
 Let $G$ be a finite group and $p\in \pi(G)$. If $G$ is $p$-solvable, then:
 \begin{itemize}
 \item[a)] $\eta(\Gamma_p(G))=|G|-r_p(G)$.
 \item[b)]  $\Gamma_p(G)$ is integral.
\item[c)]  $\mathcal{E}(\Gamma_p(G))$ is at least the maximum of
$$\{r_p(G)+c_p(G)(d_p(G)-1), \sqrt{|G|d_p(G)+r_p(G)(r_p(G)-1)}\}.$$
\item[d)] The diameter of $\Gamma_p(G)$ is at most $ |G|_p$.
 \end{itemize}
 \end{theorem}

The simple graph $\Gamma$ with $n$ vertices is said to be \textit{hyperenergetic} if its energy exceeds the energy of the complete graph $K_n$; that is, if $\mathcal{E}(\Gamma)>2n-2$. Otherwise, $\Gamma$ is called \textit{non-hyperenergetic}. Hyperenergetic graphs were first introduced by Gutman \cite{Hyper}. In \cite{182} Gutman conjectured that $\mathcal{E}(\Gamma)\leq 2n-2$ holds for all graphs with $n$ vertices. In general, this conjecture is false. The first counterexample was found in 1986 using Cvetkovics's computer system graphs \cite{3}. As another counterexample, Akbari et al. \cite{kn} proved that Kneser graphs and their complements are hyperenergetic. The following assertion shows the validity of this conjecture for a class of cayley graphs on $p$-solvable groups generated by $p$-singular elements.

\begin{theorem}\label{nil}
Assume that $G$ is a $p$-solvable group, for some prime
 $p\in \pi(G)$. If $G/O_{p^\prime}(G)$ is a
  $p$-group or a Frobenius group, then $\Gamma_p(G)$ is a non-hyperenergetic graph with
  $\mathcal{E} (\Gamma_p(G))=2|G|-2|G|_{p^\prime}$.
\end{theorem}

\begin{remark}
In general the Cayley graph on a $p$-solvable group generated by $p$-singular elements is not a non-hyperenergetic graph. For example, suppose $G$ is the symmetric group on $4$ letters. Then it is easy to see that $\mathcal{E} (\Gamma_2(G))=54>2|G|-2=46$.
\end{remark}

We now investigate a useful method to construct a singular graph.

\begin{corollary}\label{solvable}
Suppose $G$ is a solvable group. Then there exists at most one prime $p\in \pi(G)$ such that $\Gamma_p(G)$ is non-singular.
\end{corollary}

The paper is organized as follows. Block theory of finite groups plays an essential role in this paper. At Section 2, we follow \cite{nav} to introduce the notion of a $p$-block of a finite group $G$ and then we present some relevant facts on it. At Section 3, we employ the character theory of finite groups to study the nullity of Cayley graphs on finite groups generated by $p$-singular elements. Then we find several ways to construct a singular graph. Finally, at section 4, we prove our main results.
\section{$p$-blocks of finite groups}
$\noindent$In this section, we wish to present some facts on $p$-blocks of finite groups. Our definitions and notations for blocks of finite groups closely follow Navarro's book \cite{nav}. Let $R$ be the ring of algebraic integers in $\mathbb{C}$.  We fix a prime $p$ and we choose a maximal ideal $M$ of $R$ containing $pR$. Let $F=R/M$ be a field of characteristic $p$, and let $*:R\rightarrow F$ be the natural ring homomorphism. The set of all irreducible ordinary characters and irreducible Brauer characters of $G$ are denoted by $\mathrm{Irr}(G)$ and $\mathrm{IBr}(G)$, respectively. A \textit{$p$-regular} element of $G$ is an element $g\in G$ such that the order of $g$ is not divisible by $p$. The set of all $p$-regular elements of $G$ is denoted by $G^0$.

If $\chi\in \mathrm{Irr}(G)$, it is well known that $\chi$ uniquely determines an algebra homomorphism $w_\chi:Z(\mathbb{C}G)\rightarrow \mathbb{C}$. If $K$ is a conjugacy class of $G$, $x_K\in K$ and $\hat{K}=\sum_{x\in K} x$, then
$w_\chi(\hat{K})=|K|\chi(x_K)/\chi(1)$. This induces an  algebra  homomorphism $\lambda_\chi: Z(FG)\rightarrow F$ by setting $\lambda_\chi(\hat{K})=w_\chi(\hat{K})^*$, for every conjugacy class $K$ of $G$. If $\varphi \in \mathrm{IBr}(G)$ and $\mathcal{X}$ is an irreducible $F$-representation of $G$ affording $\varphi$, then for every conjugacy class $K$ of $G$, we have $\mathcal{X}(\hat{K})=\lambda_\varphi(\hat{K})I_n$. It defines an algebra homomorphism $\lambda_\varphi:Z(FG)\rightarrow F$.

The \textit{$p$-blocks} of $G$ are the equivalence classes in $\mathrm{Irr}(G)\cup \mathrm{IBr}(G)$
under the relation $\chi \sim \varphi$ if $\lambda_\chi=\lambda_\varphi$, for $\chi,\varphi\in \mathrm{Irr}(G)\cup \mathrm{IBr}(G)$. Note that the block $B$ containing $\chi$ is uniquely determined by the algebra homomorphism $\lambda_B:=\lambda_\chi$.
If $B$ is a $p$-block of $G$, $\mathrm{Irr}(B)$ denotes the set $B\cap \mathrm{Irr}(G)$.
 Also the set of $p$-blocks of $G$ is denoted by $\mathrm{Bl}(G)$.
 The \textit{principal $p$-block} of $G$, denoted by $B_0$, is the unique block of $G$ which contains the principal character $1_G$.

For a positive integer $n$ and a prime $p$, the $p$-part of $n$ is denoted by $n_p$. If $B$ is a $p$-block of $G$, the \textit{defect} of the block $B$ is defined as the integer $d(B)$ satisfying $p^{a-d(B)}=Min \{\chi(1)_p|\chi \in \mathrm{Irr}(B)\}$, where $|G|_p=p^a$. For $\chi\in \mathrm{Irr}(B)$, the \textit{height of $\chi $} is the positive integer $h$ satisfying $\chi(1)_p=p^{a-d(B)+h}$.

 We now work toward a characterization for $\mathrm{Irr}(B_0)$. Let $\chi,\psi\in \mathrm{Irr}(G)$. It is written by $\chi\leftrightarrow \psi$ if $\sum_{x\in G^0}\chi(x)\psi(x^{-1})\neq 0$.
\begin{theorem}\label{character}\cite[Theorem 3.19]{nav}
The connected components of the graph in $\mathrm{Irr}(G)$ defined by $\leftrightarrow$ are exactly the sets $\mathrm{Irr}(B)$, for $B\in \mathrm{Bl}(G)$.
\end{theorem}
 If $\chi \in \mathrm{Irr}(B)$ has height zero, by the paragraph after the proof of \cite[Lemma 3.20]{nav}, we have $\chi\leftrightarrow \psi$ for all $\psi\in \mathrm{Irr}(B)$. Now we are ready to present our characterization.

 \begin{theorem}\label{prin}
 Suppose $\chi\in \mathrm{Irr}(G)$. Then  $\chi\in \mathrm{Irr}(B_0)$, if and only if $\sum_{g\in \Omega_p(G)}\chi(g)\neq 0$.
 \end{theorem}

 \begin{proof}
 Let $\chi\in \mathrm{Irr}(B_0)-\{1_G\}$. Since $1_G\in \mathrm{Irr}(B_0)$ has height zero, by above discussion, we have $\chi \leftrightarrow 1_G$. Also using \cite[corollary 2.14]{isa}, $\sum_{g\in G}\chi(g)=0$. Thus  $\sum_{g\in \Omega_p(G)}\chi(g)\neq 0$. Conversely, if  $\sum_{g\in \Omega_p(G)}\chi(g)\neq 0$, then as $\sum_{g\in G}\chi(g)=0$, we get $\chi\leftrightarrow 1_G$ and applying Theorem \ref{character}, we are done.
 \end{proof}

 If $B$ is a $p$-block of $G$, then the \textit{defect groups} of $B$ are the $p$-subgroups of $G$ uniquely determined by $B$ by applying \cite[Theorem 4.3]{nav}. This is a $G$-conjugacy class of $p$-subgroups of $G$ which is denoted by $\delta(B)$. Fix $D\in \delta(B)$. By \cite[Theorem 4.6]{nav}, $|D|=p^{d(B)}$. Thus as the principal $p$-block $B_0$ is of defect $|G|p$, the set of all defect groups of $B_0$ is precisely the set of all Sylow $p$-subgroups of $G$.

Let $H$ be a subgroup of $G$, $b\in \mathrm{Bl}(H)$  and $\varphi \in b$. Then the algebra homomorphism $\lambda_b:Z(FH)\rightarrow F$ can be extended to a linear map $\lambda_b^G:Z(FG)\rightarrow F$ by setting $\lambda_b^G(\hat{K})= \lambda_b(\sum_{x\in K\cap H}x)$, for every conjugacy class $K$ of $G$. If $\lambda_b^G$ is an algebra homomorphism, then by \cite[Theorem 3.11]{nav}, there exists a unique block $b^G\in \mathrm{Bl}(G)$ such that $\lambda_b^G=\lambda_{b^G}$. In this case, we say that $b^G$ is defined and it is called the \textit{induced block of $b$}.

 The primitive idempotents of $Z(\mathbb{C}G)$ are $\{e_\chi:=\frac{\chi(1)}{|G|}\sum_{g\in G}\chi(g^{-1})g|\chi\in \mathrm{Irr}(G)\}$ (see \cite[Theorem 2.12]{isa}). For $B\in \mathrm{Bl}(G)$, set $f_B:=\sum_{\chi \in \mathrm{Irr}(B)}e_\chi$. If $N$ is a normal subgroup of $G$, then the group $G$ acts on $\mathrm{Bl}(N)$ by conjugation. If $\{b_1, \dots, b_t\}$ is the $G$-orbit of $b\in \mathrm{Bl}(N)$, then the central idempotent $\sum_{i=1}^t f_{b_i}\in Z(\mathbb{C}G)$ is a sum of $\{e_\chi|\;\chi\in \mathrm{Irr}(G)\}$. Thus there exist uniquely determined blocks $B_1,\dots,B_s\in \mathrm{Bl}(G)$ such that $\sum_{i=1}^tf_{b_i}=\sum_{i=1}^sf_{B_i}$. In this case, it is said that \textit{$B_i(1\leq i \leq s)$ covers $b$}.

 Let $B$ be a $p$-block of $G$. It was conjectured by Brauer that $k(B)\leq p^{d(B)}$, where $k(B):=|\mathrm{Irr}(B)|$. This is known as Brauer's $k(B)$-conjecture. In the $p$-solvable case, the conjecture reduces to the so-called "$k(GV)$-problem" (see \cite{nag}). This problem states that $k(GV)\leq |V|$, whenever for some finite group $G$ of order prime to $p$, $V$ is a faithful $F_pG$-module over the prime field $F_p$. $k(GV)$-problem has been completely solved by Gluck et al. \cite{gl}.

 \begin{theorem}\label{sol}
  Brauer's $k(B)$-conjecture is true for $p$-solvable groups.
 \end{theorem}
\section{The nullity of cayley graphs on finite groups generated by $p$-singular elements}
$\noindent$It is well known that the eigenvalues of a normal Cayley graph  $\mathrm{Cay}(G,S)$ can be expressed in terms of the irreducible characters of the group $G$ \cite[p.235]{eigen}.

\begin{theorem}\label{eigen}(\cite{2}, \cite{6}, \cite{17}, \cite{19})
The eigenvalues of a normal Cayley graph $\mathrm{Cay}(G,S)$
are given by $\eta_\chi=\frac{1}{\chi(1)}\sum_{a\in S}\chi(a)$ where
 $\chi$ ranges over $\mathrm{Irr}(G)$. Moreover,
  the multiplicity of $\eta_{\chi}$ is $\chi(1)^2$.
\end{theorem}
By Theorem \ref{eigen}, a normal Cayley graph $\mathrm{Cay}(G,S)$ is singular if and only if there is $\chi\in \mathrm{Irr}(G)$ such that $\sum_{g\in S} \chi(g)=0$. Starting from this observation, Siemons and Zalesski \cite{sz} obtained several results on the singularity of some connected normal Cayley graphs on non-abelian simple groups.  In \cite[Theorem 1.5]{sz}, it has been shown that the Cayley graph  $\Gamma_p(G)$  is singular for all finite non-abelian simple groups with some exceptions. In this section, we wish to obtain a necessary and sufficient condition for singularity of the Cayley graph $\Gamma_p(G)$ on an arbitrary finite group $G$.

\begin{theorem}\label{main}
Let $G$ be a finite group and $p\in \pi(G)$.
\begin{itemize}
\item[a)] For the group $G$, the following conditions are equivalent:
\begin{itemize}
\item[i)] $\Gamma_p(G)$ is a singular graph.
\item[ii)] $|\mathrm{Bl}(G)|\geq 2$.
\end{itemize}
\item[b)]  $\eta(\Gamma_p(G))=|G|-\sum_{\chi \in \mathrm{Irr}(B_0)}\chi(1)^2$.
\item[c)] $\Gamma_p(G)$ is integral.
\end{itemize}
\end{theorem}
\begin{proof} \textbf{a)} Suppose $\Gamma_p(G)$ is singular. Then using Theorem \ref{eigen}, there exists $\chi\in \mathrm{Irr}(G)$ such that $\sum_{g\in \Omega_p(G)} \chi(g)=0$.
 Thus by Theorem \ref{prin}, we have  $\chi \notin B_0$. Therefore $|\mathrm{Bl}(G)|\geq 2$. Conversely, assume that $|\mathrm{Bl}(G)|\geq 2$. Then by Theorem \ref{prin}, there exists $\chi \in \mathrm{Irr}(G)$ such that $\sum_{g\in \Omega_p(G)}\chi(g)=0$ and this conclude the singularity of $\Gamma_p(G)$.\\
 \textbf{b)} Using Theorems \ref{eigen} and \ref{prin}, $\eta(\Gamma_p(G))=\sum_{\chi\in \mathrm{Irr}(G)-\mathrm{Irr}(B_0)}\chi(1)^2$.  Also by \cite[Corollary 2.7]{isa}, $|G|=\sum_{\chi \in \mathrm{Irr}(G)} \chi(1)^2$. Hence $\eta(\Gamma_p(G))=|G|-\sum_{\chi \in \mathrm{Irr}(B_0)}\chi(1)^2$.\\
   \textbf{c)} Applying \cite[Theorem 1]{ko}, it is clear.
\end{proof}

Now, we briefly review some notations and definitions from representation theory of symmetric groups \cite{GA}. A \textit{partition} $lamda$ of a positive integer $n$, denoted by $\lambda \vdash n$, is a weakly decreasing finite sequence of positive integers $\lambda=(\lambda_1, \dots, \lambda_l)$ such that $n:=\sum_{i=1}^{l}\lambda_i$. The \textit{diagram} of $\lambda$ is $[\lambda]:=\{(i,j)|\;1\leq i\leq l$, $1\leq j\leq \lambda_i\}$. We call the elements of $[\lambda]$ the \textit{nodes} of $\lambda$.  The \textit{conjugate} of a partition $\lambda$ is the partition $\lambda^\prime$ whose diagram is the transpose of $[\lambda]$.
We denote by $H^\lambda_{i,j}$ the \textit{$(i,j)$-hook} of $[\lambda]$, which consists of the $(i,j)$-node, called the corner of the hook, and all the nodes to the right of it in the same row together with all the nodes lower down and in the same column as the corner. The $(i,\lambda_i)$-node is called the \textit{hand} of the hook and the $(\lambda^{\prime}_j, j)$-node
 is called the \textit{foot} of $H^\lambda_{i,j}$.
  The \textit{hook length} of a node $u=(i,j)\in [\lambda]$ is
 the number of nodes in $H^\lambda_{i,j}$.
  The \textit{rim $q$-hook} of  $[\lambda]$ corresponding to a node $u=(i,j)$ with hook length $q$ consists of the nodes on the rim between the hand and the foot
 of $H^\lambda_{i,j}$, including the hand and the foot nodes.
 The diagram $[\lambda]$ (the  partition $\lambda$) which does not contain any node of hook length $q$ is called a \textit{$q$-core}.  The \textit{$q$-core of $[\lambda]$}, denoted by $[\tilde{\lambda}]$, is a diagram obtained by successive removals of rim $q$-hooks from $[\lambda]$. Each conjugacy class of the symmetric group $S_n$ corresponds
naturally to the partitions of $n$ associated to the cycle structure of that class.   The value of the irreducible character $\chi^\alpha$, labeled by the partition $\alpha$, evaluated at the conjugacy class corresponding to a partition $\beta$ can be calculated recursively by the well known Murnaghan-Nakayama formula \cite{GA}.
\begin{lemma}\label{simple}
Suppose $S$ is a simple group and $p$ is an odd prime dividing the order of $S$. Then $\Gamma_p(S)$ is singular.
\end{lemma}
\begin{proof}
Using Theorem \ref{main} (a), it suffices to show that $S$ has a non-principal $p$-block. By \cite[Corollary 2]{gr}, except the following cases, the group $S$ has a $p$-block of defect zero which is a non-principal $p$-block of $S$. We now should consider the following cases:\\
\textbf{Case 1.} \textit{$p=3$ and $S$ is isomorphic to one of the sporadic simple groups $Suz$ ore $Co_3$:} Using \cite{atlas} and Theorem \ref{prin}, it is easy to see that $S$ has a non-principal $p$-block.\\
\textbf{Case 2.} \textit{$p=3$ and $S$ is isomorphic to the alternating group $A_n$ with $3n+1=m^2r$ where $r$ is square-free and divisible by some prime $q\equiv 2$ (mod $3$):} It is easy to see that $n\geq 7$. Suppose $n$ is not divisible by $3$.
  Then using \cite[Theorem 2.5.7]{GA}, the restriction of the irreducible character of the symmetric group $S_n$ labeled by the partition $\alpha:=(n-1,1)$ to  $A_n$ is an irreducible character. Thus as $\tilde{\alpha}\neq \tilde{1}_{A_n}$, \cite[Theorem 6.1.46]{GA} implies that the restriction of $[\alpha]$ to $A_n$ belongs to a non-principal $p$-block of $A_n$. Now let $n$ be divisible by $3$ and $\beta:=(n-2,2)$. Then again using \cite[Theorem 2.5.7]{GA} and \cite[Theorem 6.1.46]{GA}, we deduce that the restriction of the irreducible character of the symmetric group $S_n$ labeled by the partition $\beta$ to  $A_n$ is an irreducible character which is not in the principal $p$-block of $A_n$. This completes the proof.
\end{proof}

\begin{lemma}\label{nilpotent}
Let $G$ be a nilpotent group and $p\in \pi(G)$. If $G$ is not a $p$-group, then $\Gamma_p(G)$ is singular.
\end{lemma}
\begin{proof}
Assume that $P$ is a Sylow $p$-subgroup of $G$. Then as $G$ is not a $p$-group, there exists a non-trivial $p^\prime$-subgroup $H$ of $G$ such that $G\cong H \times P$. Now let $\chi\in \mathrm{Irr}(H)$ be a non-trivial irreducible character of $H$. Then $\chi \times 1_P\in \mathrm{Irr}(G)$. Applying \cite[Corollary 2.14]{isa}, it is easy to see that $\sum_{g\in G^0}(\chi \times 1_P)(g)=\sum_{h\in H}\chi(h)=0$. Hence by Theorem \ref{prin}, we deduce that $G$ has a non-principal $p$-block containing $\chi\times 1_P$. Therefore by Theorem \ref{main} (a), the graph $\Gamma_p(G)$ is singular.
\end{proof}

Suppose $G$ is a finite group. The \textit{Fitting subgroup} $\mathrm{F}(G)$ of $G$ is the unique largest normal nilpotent subgroup of $G$. For a subgroup $H$ of $G$, the \textit{sentralizer} and the \textit{normalizer} of $H$ in $G$ are defined as $\mathrm{C}_G(H):=\{g\in G|gh=hg, \, for \,all\,h\in H\}$ and $\mathrm{N}_G(H):=\{g\in G|ghg^{-1}\in H, \, for \,all\,h\in H\}$, respectively. A subgroup $H$ of the group $G$ is \textit{subnormal}, if there exists a finite chain of subgroups of the group $G$, each one normal in the next, beginning at $H$ and ending at $G$. The cyclic group of order $n$ is denoted by $C_n$. Let $p$ be a prime. A \textit{$p$-element} of $G$  is an element $g\in G$ such that the order of $g$ is a power of $p$.  Finally, the \textit{$p$-core} $O_p(G)$ of the group $G$ is the largest normal $p$-subgroup of $G$. We now apply Theorem \ref{main} to obtain new classes of singular Cayley graphs.

\begin{corollary}\label{almost}
Assume that $G$ is a finite group and $p\in \pi(G)$. If one of the following conditions holds, then $\Gamma_p(G)$ is a singular graph;
\begin{itemize}
\item[a)] There exists a non-trivial subnormal subgroup $H$ of $G$ such that either $p$ does not divide the order of $H$, or $\Gamma_p(H)$ is singular.
\item[b)] For some $p$-subgroup $ P$ of $G$, there exists a subgroup $H$ of $G$ such that $P \mathrm{C}_G(P)\subseteq H\subseteq \mathrm{N}_G(P)$ and $\Gamma_p(H)$ is singular.
\item[c)] For some $p$-element $x\in G$, the graph $\Gamma_p(\mathrm{C}_G(x))$ is singular.
\item[d)] $p$ is odd and $G$ has a non-abelian minimal normal subgroup $N$.
\item[e)] $G$ is a non-solvable group with a cyclic Sylow $p$-subgroup.
\item[f)] The Fitting subgroup $\mathrm{F}(G)$ of $G$ is not a $p$-subgroup of $G$.
\item[g)] $p$ is odd and $O_p(G)=1$.
\end{itemize}
\end{corollary}

\begin{proof} We carry the proof out by the following cases:\\
\textbf{a)} \textit{There exists a non-trivial subnormal subgroup $H$ of $G$ such that either $p$ does not divide the order of $H$, or $\Gamma_p(H)$ is singular:} Without loss of generality, we can assume that $H$ is a proper subgroup of $G$. Since $H$ is a subnormal subgroup of $G$, for some positive integer $t$, there exist distinct subgroups $H_1, H_2,\dots,H_t$ of $G$ such that $H\triangleleft H_1 \triangleleft H_2 \triangleleft \dots \triangleleft H_t=G$. We do the proof by induction on $t$. Let $t=1$. Then one of the following occurs:\\
\textbf{i)} \textit{$p$ does not divide the order of $H$:} Using \cite[Theorem 2.12]{nav}, $\mathrm{Irr}(H)=\mathrm{IBr}(H)$. Hence as $H$ is non-trivial, \cite[Theorem 3.9]{nav} implies that $|\mathrm{Bl}(H)|=|\mathrm{Irr}(H)|\geq 2$. Suppose $\chi \in \mathrm{Irr}(H)-\{1_H\}$ and $B_0$ is the principal $p$-block of $G$. Then by \cite[Theorem 9.2]{nav}, we deduce that $B_0$ does not cover the block $b:=\{\chi\}$ of $H$. Thus there exists a non-principal block $B$ of $G$ such that $B$ covers $b$. Therefore $|\mathrm{Bl}(G)|\geq 2$ and Theorem \ref{main} (a) completes the proof.\\
\textbf{ii)} \textit{$\Gamma_p(H)$ is singular:} By Theorem \ref{main}, $|\mathrm{Bl}(H)|\geq 2$. Hence $H$ has a non-principal block $b$. Using \cite[Teorem 9.2]{nav}, the block $b$ is not covered by the principal block of $G$. Thus $|\mathrm{Bl}(G)|\geq 2 $ and again, Theorem \ref{main} (a) completes  the proof.\\
Now assume that $k\geq 2$ is a positive integer, $t=k$ and the assertion is true for $k-1$. If $p$ divides the order of $H_{k-1}$, then by induction hypothesis and this fact that $H$ is a non-trivial subnormal subgroup of $H_{k-1}$, we have $\Gamma_p(H_{k-1})$ is singular. Therefore applying the initial case for $H_{k-1}$ and $G$, we deduce that $\Gamma_p(G)$ is singular.\\
\textbf{b)}  \textit{For some $p$-subgroup $ P$ of $G$, there exists a subgroup $H$ of $G$ such that $P \mathrm{C}_G(P)\subseteq H\subseteq \mathrm{N}_G(P)$ and $\Gamma_p(H)$ is singular:} By Theorem \ref{main} (a), the group $H$ has a non-principal $p$-block $b$. Also using \cite[Theorem 4.14]{nav} and Brauer's third main theorem \cite[Theorem 6.7]{nav}, we deduce that $B:=b^G$ is a non-principal $p$-block of $G$. Hence $|\mathrm{Bl}(G)|\geq 2$ and by Theorem \ref{main}, we are done. \\
\textbf{c)} \textit{For some $p$-element $x\in G$, the graph $\Gamma_p(\mathrm{C}_G(x))$ is singular:} By part (b), it is straight forward.\\
\textbf{d)} \textit{$p$ is odd and $G$ has a non-abelian minimal normal subgroup $N$:} If $p$ does not divide the order of $N$, by  part (a), we are done. Hence we can assume that $p$ divides the order of $N$.
Since $N$ is non-abelian, there exists a non-abelian simple subgroup $S$ of $N$ such that $N\cong S^k$, for some positive integer $k$.
 Thus as $S$ is a subnormal subgroup of $G$, using Lemma \ref{simple} and  part (a), we deduce that $\Gamma_p(G)$ is singular.\\
\textbf{e)}\textit{ $G$ is a non-solvable group with a cyclic Sylow $p$-subgroup:} Using \cite{mich} and Theorem \ref{main} (a), we have nothing to prove. \\
\textbf{f)} \textit{The Fitting subgroup $\mathrm{F}(G)$ of $G$ is not a $p$-subgroup of $G$:} By Lemma \ref{nilpotent} and part (a), we are done.\\
\textbf{g)} \textit{$p$ is odd and $O_p(G)=1$:} Let $N$ be a minimal normal subgroup of $G$. Since $O_p(G)$ is the largest normal $p$-subgroup of $G$, the group $N$ is not a $p$-subgroup of  $G$. Thus either $N$ is non-abelian or for some prime $q\neq p$, the group $N$ is a non-trivial elementary abelian $q$-group. In the later case, part (a) implies that the graph $\Gamma_p(G)$ is singular. Now suppose $N$ is non-abelian. Then using part (d), we are done.  \end{proof}
\begin{remark}\label{r}
Assume that $G$ is the sporadic simple group $M_{22}$. Using  \cite{atlas}, we can see that the graph $\Gamma_2(G)$ is not singular. This shows that the assumption "$p$ is odd" is necessary in Corollary \ref{almost} (d).
\end{remark}

To determine the significance of these results, we wish to construct a so large class of integral singular graphs. For this goal, we associate an integral singular graph to each pair $(G,p)$, where $G$ ranges over all finite groups, and $p$ over all prime divisors of the order of $G$.
\begin{corollary}\label{large}
Suppose $G$ is a finite group. Then for every prime $p\in \pi(G)$, the graph $\Gamma_p(G\times C_6)$ is singular.
\end{corollary}
\begin{proof} If $p=2$, then as the group $G\times C_6$ has a nilpotent normal subgroup isomorphic to $ C_6$, using Lemma \ref{nilpotent} and Corollary \ref{almost} (a), we are done. Thus we can assume that $p$ is odd. If $O_p(G)=1$, then by Corollary \ref{almost} (g), $\Gamma_p(G\times C_6)$ is singular. Now let  $O_p(G)\neq 1$. It is clear that $O_p(G) \times C_2$ is a nilpotent normal subgroup of $G$ which is not a $p$-group. Hence applying Lemma \ref{nilpotent} and Corollary \ref{almost} (a), we deduce that $\Gamma_p(G\times C_6)$ is singular.
\end{proof}
\section{Proof of our main results}
\noindent In this section, we wish to prove our main results.  Assume that $N$ is a normal subgroup of a group $G$.  When we work with $G/N$, we can identify $\mathrm{Irr}(G/N)$ with $\{\chi \in \mathrm{Irr}(G)|\, N \subseteq \mathrm{Ker}(\chi)\}$, where for every $\chi \in \mathrm{Irr}(G)$, the \textit{kernel of $\chi$} is defined as $\mathrm{Ker}(\chi):=\{g\in G|\, \chi(g)=\chi(1)\}$. For every $\chi \in \mathrm{Irr}(G)$, we set $\lambda_\chi(G,p):=1/\chi(1)|\sum_{x\in \Omega_p(G)}\chi(x)|$. We begin with the following observation.

\begin{lemma}\label{ener}
Let $G$ be a $p$-solvable group, for some $p\in \pi(G)$. Then
$$\mathcal{E}(\Gamma_p(G))=\sum_{\chi\in \mathrm{Irr}(G/O_{p^\prime}(G))}\chi(1)^2\lambda_\chi(G,p).$$
\end{lemma}

\begin{proof}
 By Theorems \ref{eigen} and \ref{prin}, we have
\begin{align}
\mathcal{E}(\Gamma_p(G))&=\sum_{\chi \in \mathrm{Irr}(G)}\chi(1)^2 \lambda_\chi(G,p)\nonumber\\
&=\sum_{\chi\in \mathrm{Irr}(G)-\mathrm{Irr}(B_0)}\chi(1)^2\lambda_\chi(G,p)+ \sum_{\chi\in \mathrm{Irr}(B_0)}\chi(1)^2\lambda_\chi(G,p)\nonumber\\
&=\sum_{\chi\in \mathrm{Irr}(B_0)}\chi(1)^2\lambda_\chi(G,p).\nonumber
\end{align}
Applying \cite[Theorem 10.20]{nav}, we deduce that $\mathrm{Irr}(B_0)$ is the set of all irreducible characters of $G$ lying over the trivial character  $1_{O_{p^\prime}(G)}$ and thus $\mathrm{Irr}(B_0)=\mathrm{Irr}(G/O_{p^\prime}(G))$. Hence
$$\mathcal{E}(\Gamma_p(G))=\sum_{\chi\in \mathrm{Irr}(G/O_{p^\prime}(G))}\chi(1)^2\lambda_\chi(G,p)$$
\end{proof}
Now we are ready to prove Theorem \ref{energy}.\\
\noindent\textit{Proof of Theorem \ref{energy}:} a) Applying \cite[Theorem 10.20]{nav}, we deduce that $\mathrm{Irr}(B_0)=\mathrm{Irr}(G/O_{p^\prime}(G))$. Also by \cite[Corollary 2.7]{isa}, $\sum_{\chi \in \mathrm{Irr}(G/O_{p^\prime}(G))}\chi(1)^2=r_p(G)$. Thus using Theorem \ref{main} (b), we are done.\\
b) Using Theorem \ref{main} (c), $\Gamma_p(G)$ is integral.\\
c) Since $\Gamma_p(G)$ is a graph with precisely $c_p(G)$ connected components, using Perron-Frobenius Theorem, we deduce that $d_p(G)$ is an eigenvalue of $\Gamma_p(G)$ with multiplicity $c_p(G)$. Now suppose $M:=\{\chi\in \mathrm{Irr}(G/O_{p^\prime}(G))|\,\lambda_\chi(G,p)=d_p(G)\}$. Using Lemma \ref{ener}, we have
\begin{align}\mathcal{E}(\Gamma_p(G))&=\sum_{\chi\in \mathrm{Irr}(G/O_{p^\prime}(G))}\chi(1)^2\lambda_\chi(G,p)\nonumber\\
&=\sum_{\chi \in M}\chi(1)^2d_p(G)+ \sum_{\chi \in \mathrm{Irr}(G/O_{p^\prime}(G))-M}\chi(1)^2\lambda_\chi(G,p).\nonumber
\end{align}
Now let $\chi \in \mathrm{Irr}(G/O_{p^\prime}(G))-M$. Applying Theorems \ref{eigen} and \ref{prin}, we deduce that $\lambda_\chi(G,p)$ is non-zero. Thus by part (b), we have $\lambda_\chi(G,p)$ is a positive integer. Hence using \cite[Corollary 2.7]{isa},
\begin{align}
\mathcal{E}(\Gamma_p(G))&\geq d_p(G)c_p(G)+ \sum_{\chi \in \mathrm{Irr}(G/O_{p^\prime}(G))-M}\chi(1)^2\nonumber\\
&=c_p(G)(d_p(G)-1)+r_p(G).
\end{align}

Since $\Gamma_p(G)$ is a $d_p(G)$-regular graph of order $|G|$, it is clear that the graph $\Gamma_p(G)$ has $|G|d_p(G)/2$ edges. Hence applying part (a) and \cite[Corollary 4]{An1}, we deduce that
\begin{align}
\mathcal{E}(\Gamma_p(G))\geq \sqrt{|G|d_p(G)+r_p(G)(r_p(G)-1)}.
\end{align}

Hence using inequalities (1) and (2), we conclude that $\mathcal{E}(\Gamma_p(G)$ is at least the maximum of
$$\{c_p(G)(d_p(G)-1)+r_p(G), \sqrt{|G|d_p(G)+r_p(G)(r_p(G)-1)}\}.$$
d) Applying Theorem \ref{main} (b) and \cite[Theorem 3.13]{Cv}, the diameter of $\Gamma_p(G)$ is at most $|\mathrm{Irr}(B_0)|$. Hence as every Sylow $p$-subgroup of $G$ is a defect group of $B_0$, Theorem \ref{sol} completes the proof.\qed\\

 If Brauer's $k(B)$-conjecture (see sect. 2) is true for principal blocks of finite groups, then we can state Theorem \ref{energy} (d) for every finite group. Thus we can propose the following conjecture;

\begin{conjecture}
Suppose $G$ is a finite group. Then for every prime $p\in \pi(G)$, the diameter of $\Gamma_p(G)$ is at most $|G|_p$.
\end{conjecture}
Now we prove Theorem \ref{nil}.\\
\textit{Proof of Theorem \ref{nil}:}
We claim that $\Omega_p(G)=PO_{p^\prime}(G)-O_{p^\prime}(G)$, where $P$ is a Sylow $p$-subgroup of $G$. It is easy to see that $PO_{p^\prime}(G)-O_{p^\prime}(G)\subseteq \Omega_p(G)$. Thus it suffices to show that $\Omega_p(G)\subseteq P O_{p^\prime}(G)-O_{p^\prime}(G)$. On the contrary, suppose $x\in \Omega_p(G)-( P O_{p^\prime}(G)-O_{p^\prime}(G))$. Then there exist a non-trivial $p$-element $x_p$ and a non-trivial $p^\prime$-element $x_{p^\prime}$ such that $x=x_px_{p^\prime}=x_{p^\prime} x_p$.
Since $G/O_{p^\prime}(G)$ is a
  $p$-group or a Frobenius group, $ P O_{p^\prime}(G)$ is a normal subgroup of $G$. If $x_{p^\prime}\in O_{p^\prime}(G)$, then as every Sylow $p$-subgroup of $G$ is contained in $ P O_{p^\prime}(G)$, we deduce that $x=x_px_{p^\prime}\in P O_{p^\prime}(G)$ which is a contradiction. Thus $x_{p^\prime}\notin O_{p^\prime}(G)$.
   Hence $x_{p^\prime}O_{p^\prime}(G)$ is a non-trivial element of $ C_{G/O_{p^\prime}(G)}(x_pO_{p^\prime}(G))$ which is a contradiction with the structure of $G/O_{p^\prime}(G)$ (see \cite[Theorem 6.4]{Is}). Therefore  $\Omega_p(G)=PO_{p^\prime}(G)-O_{p^\prime}(G)$. Clearly, $H_p(G)=PO_{p^\prime}(G)$. Since $\Gamma_p(G)$ is a $d_p(G)$-regular graph with precisely $c_p(G)$ connected components, applying Perron-Frobenius Theorem, we deduce that $d_p(G)=|H_p(G)|-|O_{p^\prime}(G)|$ is an eigenvalue of $\Gamma_p(G)$ with multiplicity $c_p(G)$. Now let $M:=\{\chi\in \mathrm{Irr}(G/O_{p^\prime}(G))|\,\lambda_\chi(G,p)=d_p(G)\}$.
 Thus by Lemma \ref{ener}, we have
\begin{align}
\mathcal{E}(\Gamma_p(G))&=\sum_{\chi\in \mathrm{Irr}(G/O_{p^\prime}(G))}\chi(1)^2\lambda_\chi(G,p)\nonumber\\
&=|PO_{p^\prime}(G)-O_{p^\prime}(G)|c_p(G) +\sum_{\chi\in \mathrm{Irr}(G/O_{p^\prime}(G))-M}\chi(1)^2\lambda_{\chi}(G,p)\nonumber\\
&=( |H_p(G)|-|O_{p^\prime}(G)|)c_p(G)+\sum_{\chi\in \mathrm{Irr}(G/O_{p^\prime}(G))-M}\chi(1)^2\lambda_{\chi}(G,p)\nonumber\\
&=|G|-|G|_{p^\prime}+\sum_{\chi\in \mathrm{Irr}(G/O_{p^\prime}(G))-M}\chi(1)^2\lambda_{\chi}(G,p).\nonumber
\end{align}
Now let $\chi\in \mathrm{Irr}(G/O_{p^\prime}(G))-M$ and $y\in O_{p^\prime} (G)$. If $P\subseteq \mathrm{Ker}(\chi)$, then $\Omega_p(G)\subseteq \mathrm{Ker}(\chi)$ and we get $\lambda_\chi(G,p)=d_p(G)$ which is a contradiction with this fact that $\chi \notin M$. Thus $P\nsubseteq \mathrm{Ker}(\chi)$. Hence using \cite[Corollary 2.14]{isa} and this fact that $O_{p^\prime}(G)\subseteq \mathrm{Ker}(\chi)$, we deduce that $\sum_{z\in Py-\{1\}}\chi(z)/\chi(1)=-1$. Thus $\lambda_{\chi}(G,p)=|O_{p^\prime}(G)|$. Therefore applying \cite[Corollary 2.7]{isa}, we have
\begin{align}
\mathcal{E}(\Gamma_p(G))&= |G|-|G|_{p^\prime}+|O_{p^\prime}(G)|\sum_{\chi\in \mathrm{Irr}(G/O_{p^\prime}(G))-M} \chi(1)^2\nonumber\\
&=|G|-|G|_{p^\prime}+|O_{p^\prime}(G)|(r_p(G)-c_p(G))\nonumber\\
&=2|G|-2|G|_{p^\prime}\nonumber.
\end{align}
Hence $\Gamma_p(G)$ is a non-hyperenergetic graph and this completes the proof.
\qed\\

We end this section by the proof of Corollary \ref{solvable}.\\
\textit{Proof of Corollary \ref{solvable}:} On the contrary, suppose $p,q\in \pi(G)$ are two distinct primes such that $\Gamma_p(G)$ and $\Gamma_q(G)$ are non-singular. Applying Theorem \ref{energy} (a), we deduce that $O_{p^\prime}(G)=O_{q^\prime}(G)=1$. Without loss of generality, we can assume that $p$ is odd. Since $O_{q^\prime}(G)=1$, we have $O_{p}(G)=1$. Hence using Theorem \ref{main} (g), $\Gamma_p(G)$ is singular which is a contradiction.\qed


\section*{Acknowledgements}
 I would like to express
my gratitude to Prof. Gunter Malle for his valuable comments.\\
Funding: This research was supported in part
by a grant  from School of Mathematics, Institute for Research in Fundamental Sciences (IPM).



\begin{thebibliography}{22}
\bibitem{Ak}
S. Akbari, M.A. Hosseinzadeh, A short proof for graph energy is at least twice of minimum degree, MATCH Commun. Math. Comput. Chem. 83 (2020) 631-633.
%
\bibitem{kn}
S. Akbari, F. Moazami, S. Zare, Kneser graphs and their complements are hyperenergetic, MATCH Commun. Math. Comput. Chem. 61 (2009) 361-368.
%
\bibitem{al}
A. Al-Tarimshawy, J. Siemons, Singular graphs with dihedral group action, Discrete Math. 344(1) (2021) Paper No. 112119, 6 pp
%
\bibitem{An1}
E. Andrade, J. R. Carmona, G. Infante, M. Robbiano, A lower bound for the energy of hypoenergetic and non hypoenergetic graphs, MATCH Commun. Math. Comput. Chem. 83 (2020) 579-592.
%
\bibitem{At}
P. Atkins, J. de Paula, Physical Chemistry, eighth ed. Oxford University Press, 2006.

\bibitem{2}
L. Babai, Spectra of Cayley graphs, J. Comb. Theory, Ser. B 27(1979) 180-189. DOI: https://doi.org/1001016/0095-8956(79)-90079-0
%
\bibitem{Bo}
S. B. Bozkurt Altindag, D. Bozkurt, Lower bounds for the energy of (bipartite) graphs, MATCH Commun. Math. Comput. Chem. 77 (2017) 9-14.
%
\bibitem{atlas}
J.H. Conway, R.T. Curtis, S.P. Norton, R.A. Parker, R.A. Wilson, Atlas of finite groups, Clarendon Prees, Oxford 1985.
%
\bibitem{eigen}
C. W. Curtis, I. Reiner, Representation theory of finite groups and associative algebras, Pure and Applied Mathematics, vol. XI, Wiley, New york, 1962.
%
\bibitem{Cv}
D.M. Cvtkovi\'{c}, M. Doob, H. Sachs, Spectra of Graphs-Theory and Application, 3rd edition, Johann Ambrosius Barth Verlag (Heidelberg), 1995.
%
\bibitem{3}
D.M. Cvtkovi\'{c}, I. Gutman, The computer system GRAPH: A useful toll in chemical graph theory, J. Comput. Chem. 7 (1986) 640-644.
%
\bibitem{Cvg}
D.M. Cvtkovi\'{c}, I. Gutman, The algebraic multiplicity of the number zero in the spectrum of a bipatite graph, Matematicki Vesnik (Beograd) 9 (1972) 141-150.
%
\bibitem{6}
P. Diaconis, M. Shahshahani, Generating a random permutation with random transpositions, Z. Wahrsch. Verw. Gebiete. 57(1981) 159-179. 0044-3719/81/0057/0159/S04.20
%
\bibitem{gl}
D. Gluck, K. Magaard, U. Riese, P. Schmid, The solution of the $k(GV)$-problem, J. Algebra, 279(2) (2004), 694-719.
%
\bibitem{gr}
A. Granville, K. Ono, Defect zero $p$-blocks for finite simple groups, Trans. Amer. Math. Soc. 348(1) (1996), 331-347.
%
 \bibitem{Huckel}
 A. Graovac, I. Gutman, N. Trinajste, T. Ivkovi, Graph theory and molecular orbitals. Theoretica chimica acta, 26(1) (1972) 67-78.
%
\bibitem{8}
I. Gutman, Topology and stability of canjugated hydrocarbons. The dependence of total $\pi$-electron energy on molecular topology, J. Serb. Chem. Soc. 70 (2005) 441-456.
%
\bibitem{Hyper}
I. Gutman, Hyperenergetic Molecular graphs, J. Serb. Chem. Soc. 64 (1999) 199-205.
%
\bibitem{182}
I. Gutman, The energy of a graph, Ber. Math. Statist. Sekt. Forschungszentrum Graz 103 (1978) 1-22.
%
\bibitem{7}
I. Gutman, On graphs whose energy exceeds the number of vertices, Linear Algebra Appl. 429 (2008) 2670-2677.
%
\bibitem{ph}
I. Gutman, B. Borovicanin, Nullity of graphs: an updated survey, Zb. Rad. (Beogr.) 14(22) (2011) 137-154.
%
\bibitem{GZ}
I. Gutman, S. Zare firoozabadi, J. A. de la Pena, J. Rada, On the energy of regular graphs, MATCH Commun. Math. Comput. Chem. 57 (2007) 435-442.
%
\bibitem{Is}
I. M. Isaacs, Finite group theory, American Mathematical Socitey, Providence,RI,2005.
%
\bibitem{isa}
I.M. Isaacs, Character Theory of Finite Groups, AMS Chelsea Publishing, Providence, RI, 2006. Corrected reprint of the 1976 original [Academic Press. New York; MR0460423].
%
\bibitem{jahan}
A. Jahanbani, Some new lower bounds for energy of graphs, Appl. Math. Comput. 296 (2017), 233-238.
%
\bibitem{GA}
 G. James, A. Kerber.  The Representation Theory of the Symmetric Group. Addison-Wesley Publishing Company, 1981.
%
\bibitem{ko}
E. V. Konstantinova, D. Lytkina, Integral cayley graphs over finite groups, Algebra Colloq. 27(1) (2020) 131-136.
%
\bibitem{17}
A. Lubotzky, Discrete groups, Expanding graphs, and invariant Measures, Birkhauser Verlag, Basel, 1994.
%
\bibitem{Ma}
C. A. Marin, J. Monsalve, J. Rada, Maximum and minimum energy trees with two and three branched vertices, MATCH Commun. Math. Comput. Chem.74 (2015) 285-306.
%
\bibitem{mich}
G. Michler, Nonsolvable finite groups with cyclic sylow $p$ subgroups have nonprincipal $p$-blocks, J. Algebra 83 (1) (1983) 179-188.
%
\bibitem{nag}
H. Nagao, On a conjecture of Brauer for $p$-solvable groups, J. Math. Osaka City Univ. 13 (1962) 35-38.
%
\bibitem{nav}
G. Navarro, Characters and blocks of finite groups. London Mathematical Society Lecture NOte Seriety Lecture Note Series, vol. 250. Cambridge University Press, Cambridge (1998).
%
\bibitem{oboudi}
M .R. Oboudi, A very short proof for a lower bound for energy of graphs, MATCH Commun. Math. Comput. Chem. 84 (2) (2020) 345-347.
%
\bibitem{19}
M. Ram Murty, Ramanujan graphs, J. Ramanujan Math. soc. 18(2003) 1-20.
%
\bibitem{sz}
J. Siemons, A. Zalesski, Remarks on singular Cayley graphs and vanishing elements of simple groups, J. Algebraic Combin. 50(4) (2019) 379-401.
%
\bibitem{ch}
A. Streitwieser, Molecular Orbital Theory for Organic Chemists, Wiley, New York (1961).
%
\bibitem{Ti}
T. Tian, W. Yan, S. Li, On the minimal energy of trees with a given number of vertices of odd degree, MATCH Commun. Math. Comput. Chem. 73 (2015) 3-10.
%
\bibitem{zh}
B. Zhou, Lower bounds for the energy of quadrangle-free graphs, MATCH Commun. Math. Comput. Chem. 55 (2006) 91-94.
\end{thebibliography}
\end{document}